\documentclass[a4paper,11pt]{amsproc}
\usepackage{amsmath,amssymb,amsthm,amsaddr}

\usepackage{commath, color}
           
\usepackage[margin=2cm]{geometry}

\usepackage[table]{xcolor}
\usepackage{listings,booktabs}
\usepackage{caption}
\usepackage{subfigure}
\usepackage{listings}

\usepackage{listings}

\definecolor{dkgreen}{rgb}{0,0.6,0}
\definecolor{gray}{rgb}{0.5,0.5,0.5}
\definecolor{mauve}{rgb}{0.58,0,0.82}

\usepackage{verbatim,tikz}
\usepackage{hyperref,url} 
\usetikzlibrary{intersections,arrows,decorations.pathmorphing,backgrounds,positioning,fit,petri,calc}
\hypersetup{citecolor=magenta, linkcolor=blue, colorlinks=true}

\newcommand\PSL{\mathrm{PSL}}
\newcommand\PSp{\mathrm{PSp}}
\newcommand\PSU{\mathrm{PSU}}
\newcommand\SU{\mathrm{SU}}

\newcommand\PGL{\mathrm{PGL}}
\newcommand\Aut{\mathrm{Aut}}

\newcommand\GQ{\mathsf{GQ}}
\newcommand\LSce{\mathsf{LSce}}
\newcommand\AtlasRep{\textsf{AtlasRep}}

\renewcommand\leq{\leqslant}
\renewcommand\geq{\geqslant}

\newtheorem{theorem}{Theorem}[section]
\newtheorem{corollary}{Corollary}[theorem]
\newtheorem{lemma}{Lemma}[section]

\newtheorem*{conjecture*}{Conjecture}
\newtheorem*{proposition*}{Proposition}

\theoremstyle{definition}

\theoremstyle{remark}

\title[No almost simple sporadic group acts point-primitively on a GQ]{No almost simple sporadic group acts primitively\\ on the points of a generalised quadrangle}

\author{John Bamberg and James Evans}
\address{Centre for the Mathematics of Symmetry and Computation,\\ The University of Western Australia, Australia.}
\email[Bamberg]{john.bamberg@uwa.edu.au}
\email[Evans]{James.Evans10@uon.edu.au}

\subjclass[2010]{51E12,20B05,20B15,20B25}
\keywords{generalised quadrangles, hemisystems, m-ovoids, sporadic groups, primitive permutation groups}

\thanks{The research for this paper was primarily conducted as part of an honours project completed at the University of Western Australia (with first author as supervisor and second as student). The second author acknowledges the support of the Winthrop Scholarship in completing his studies.
Both authors thank Emilio Pierro for helpful discussions on some of the material in this paper.}

\begin{document}

\begin{abstract}
A \textit{generalised quadrangle} is a point-line incidence geometry $\mathcal{G}$ such that: (i) any two points lie on at most one line, and (ii) given a line $L$ and a point $p$ not incident with $L$, there is a unique point on $L$ collinear with $p$. They are a specific case of the \textit{generalised polygons} introduced by Tits \cite{Tits:1959cl}, and these structures and their automorphism groups are of some importance in finite geometry. An integral part of understanding the automorphism groups of finite generalised quadrangles is knowing which groups can act primitively
on their points, and in particular, which almost simple groups arise as automorphism groups.
We show that no almost simple sporadic group can act primitively on the points of a finite (thick) generalised quadrangle. 
We also present two new ideas contributing towards analysing point-primitive groups acting on generalised quadrangles. The first is the outline and implementation of an algorithm for determining whether a given group can act primitively on the points of some generalised quadrangle. The second is the discussion of a conjecture resulting from observations made in the course of this work: any group acting primitively on the points of a generalised quadrangle must either act transitively on lines or have exactly two line-orbits, each containing half of the lines.
\end{abstract}

\maketitle

\vspace{-0.8cm}
\section{Introduction}

A \textit{generalised $n$-gon} (or \textit{generalised polygon}) may be defined as a point-line incidence geometry whose incidence graph has diameter $n$ and girth $2n$. In addition, we require that every line is incident with at least three points, and each point is incident with at least three lines. (These are the so-called \textit{thick} generalised polygons). By the theorem of Feit and Higman \cite{Feit:1964os}, a (thick) finite generalised $n$-gon exists only for $n=2,3,4,6$ and $8$ \cite{Van-Maldeghem:1998sj}.

The generalised polygons were first introduced by Tits in \cite{Tits:1959cl}, in order to study the Lie type groups by associating them to certain geometric structures. The generalised polygons are a particular case of Tits' more general theory of \textit{buildings}. In particular, an equivalent definition of a generalised polygon is that it is an irreducible spherical building of rank $2$. Tits \cite{TitsBuildings} classified the spherical buildings of rank $3$ or more, and leaving the generalised polygons as the final missing piece in our understanding of spherical buildings.

The polygons which arise from Lie type groups came to be known as the \textit{classical generalised polygons}\footnote{Exactly which polygons are classical varies between authors. This paper only includes the \textit{Moufang polygons}, whereas in some cases the Moufang polygons and their duals are called classical.}. These geometries are often important in their own right. The classical triangles are exactly the Desarguesian projective planes (related to $\PSL(3,q)$ groups), while the classical generalised quadrangles are classical projective polar spaces (related to $\PSp(4,q)$, $\PSU(4,q)$ and $\PSU(5,q)$). The classical hexagons were discovered by Tits in \cite{Tits:1959cl}, and they are related to the $G_{2}(q)$ and ${}^{3}D_{4}(q^{3})$ groups, while the classical octagons correspond to the ${}^{2}F_{4}(2^{2k+1})$ groups.

There are non-classical generalised polygons, so not all generalised polygons can be studied using the power of Lie-group theory (see e.g. \cite{Van-Maldeghem:1998sj}, \cite{PayneThas:2009}). However the connection to group theory, via the automorphism groups of the polygons, remains their most interesting and well studied aspect. A major theme in the literature has been to explicitly classify all of the examples of generalised polygons whose automorphism groups obey certain conditions. One of the first general results of this type was by Fong and Seitz \cite{FongSeitz}
who classified the finite \textit{Moufang} generalised polygons. 
Using the Classification of Finite Simple Groups, Buekenhout and Van Maldeghem  \cite{Buekenhout:1994zp} relaxed
the Moufang condition considerably to having an automorphism group that acts distance-transitively on the points of a generalised polygon.
A non-classical generalised quadrangle appears in this classfication: the unique generalised quadrangle $\GQ(3,5)$ of order $(3,5)$.

\begin{theorem}[Buekenhout and Van Maldeghem, \cite{Buekenhout:1994zp}]
A generalised polygon with point-distance-transitive automorphism group is classical, dual classical, or $\GQ(3,5)$.
\end{theorem}

The classical and dual classical polygons have automorphism groups which are point and line-primitive, flag-transitive\footnote{That is, acting transitively on incident point-line pairs.} as well as point and line-distance-transitive. On the other hand, $\GQ(3,5)$ is point-primitive and flag-transitive, but not line-primitive. For generalised triangles/projective planes it is conjectured (see e.g., \cite{KantorPrimitivePlanes}) that a point-primitive automorphism group implies that the projective plane is Desarguesian. It is known that a counterexample would have to have a small solvable group for its automorphism group.
For generalised hexagons and octagons all classical examples have point-primitive automorphism groups and it is expected that no other hexagons or octagons exist. 
In \cite{Buekenhout:1994zp}, the authors say that they anticipated that their work would stimulate research into classifying all of the polygons obeying these weaker conditions, and this has been so. For example, some have studied what the automorphism group can be when it acts flag-transitively and point and line-primitively:

\begin{theorem}[Schneider, Van Maldeghem \cite{Schneider:2008lr}]
If $G$ acts point and line-primitively and flag-transitively on a (thick) generalised hexagon or octagon then $G$ is almost simple of Lie type.
\end{theorem}

This result was preceded by Buekenhout and Van Maldeghem's \cite{Buekenhout:1993bf}
 classification of the \textit{Atlas-groups} that act point-transitively on a finite generalised hexagon or octagon.
This includes all of the sporadic simple groups, and the conclusion is that only the classical examples (of Atlas-groups) arise,
and no almost simple group with sporadic socle can act point-transitively on a finite generalised hexagon or octagon.
An analogue for generalised quadrangles appears in \cite{Bamberg:2012yf} where it was shown computationally that if $G$ acts primitively on the points and lines of a generalised quadrangle, then the socle of $G$ is not a sporadic simple group.
Moreover, we have the following result:

\begin{theorem}[Bamberg, Giudici, Morris, Royle, Spiga \cite{Bamberg:2012yf}]
If $G$ acts point and line-primitively on a generalised quadrangle then $G$ is almost simple. If $G$ is also flag-transitive, then $G$ is almost simple of Lie type.
\end{theorem}

%
%
%
%

Only for generalised quadrangles is it expected for there to be non-classical examples with point-primitive automorphism groups, and it is also the case where the least progress has been made. The strongest result so far is that 
if $G$ acts point-primitively as automorphisms of a generalised quadrangle, then $G$ is not of \textit{Holomorph Compound} O'Nan-Scott type \cite{Bamberg2018Quadrangles}.
On the basis of the known examples it is expected that additionally all except the `Holomorph Affine' and `Almost Simple' types can be eliminated. 
Much work has already been done in classifying what happens in the affine case: the (thick) generalised quadrangles that admit an automorphism group that is line-transitive and point-primitive of Holomorph Affine type, are known \cite{PsuedoHyperovals}. The almost simple case requires much more work.

The contribution of the current paper is to describe an algorithmic way to determine explicitly whether individual groups can act point-primitively on some generalised quadrangle.
To this end, and to demonstrate the effectiveness of the method that was developed, the algorithm was implemented in the \textsf{GAP} computer algebra software \cite{GAP4}, and used to extend the results of \cite{Buekenhout:1993bf} and \cite{Bamberg:2012yf}:

\begin{theorem}\label{maintheorem}
No almost simple sporadic group can act primitively on the points of any generalised quadrangle.
\end{theorem}

Finally we conclude with a conjecture of our own, which we hope will stimulate further inquiry.
%

\begin{conjecture*} 
Suppose that $G$ is a group which acts point-primitively on a generalised quadrangle $\mathcal{Q}$. Then $G$ either acts transitively on the lines of $\mathcal{Q}$, or has only two line-orbits, both of equal size. 
\end{conjecture*}

In the case where there are two line-orbits, each orbit is a \emph{hemisystem} of the lines of $\mathcal{Q}$: half of the lines at each point are in the given orbit. A consequence of this conjecture, if true, is that if $\mathcal{Q}$ has order $(s,t)$, and $t$ is even, then $G$ is line-transitive (because
the number $(t+1)(st+1)$ of lines would be odd). So it would imply that if $t$ is even, then $G$ is not of Holomorph of Simple Group type, by
\cite[Theorem 1.1]{Bamberg:2017ab}.

\section{Finite generalised quadrangles} \label{2}

Throughout this paper, $\mathcal{Q}$ will refer to a (finite, thick) generalised quadrangle with point set $\mathcal{P}$ and line set $\mathcal{L}$. Generalised quadrangles are generalised $n$-gons with $n=4$. That is, they are point-line incidence geometries whose incidence graph has diameter $4$ and girth $8$. However, it is often useful to restate this definition in terms of a more geometrically meaningful, but equivalent, pair of axioms:

\begin{enumerate}
\item Any two points lie on at most one line, and 
\item Given a point $p$ and a line $L$ not incident, there is a unique point on $L$ collinear with $p$.
\end{enumerate}

The dual of an incidence geometry $\mathcal{G}$ is another geometry $\mathcal{G}^{D}$ with the points and lines swapped around, but the same pattern of incidence. The definition of a generalised quadrangle is duality invariant: the dual of a generalised quadrangle is another generalised quadrangle.

A generalised quadrangle $\mathcal{Q}$ is said to have an order $(s,t)$ if every line in it is incident with $s+1$ points and each point in it is incident with $t+1$ lines. The following lemma summarises the basic information about the order of a generalised quadrangle:

\begin{lemma}[{\cite[1.2.1, 1.2.2, 1.2.3]{PayneThas:2009}}]\label{basiclemma}
Let $\mathcal{Q}$ be a thick generalised quadrangle. Then:
\begin{itemize}
\item $\mathcal{Q}$ has an order $(s,t)$ for some $s,t \geq 2$,
\item $\abs{\mathcal{P}}=(s+1)(st+1)$, $\abs{\mathcal{L}}=(t+1)(st+1)$,
\item $s\leq t^{2}$, $t\leq s^{2}$,
\item $s+t$ divides $st(st+1)$.
\end{itemize}
\end{lemma}

Another important fact about generalised quadrangles is that many of the graphs associated to them, e.g., the incidence graph \cite{DistanceRegularGraphs}, must be \textit{distance-regular}. Most importantly for us is the point-collinearity graph (or dually, the line-concurrency graph). The point-collinearity graph of a quadrangle (or any point-line geometry) is the graph with the point set as vertices, and points adjacent if they are distinct and collinear in the geometry. It follows from \cite[1.2.2]{PayneThas:2009} that the point-collinearity graph of a  quadrangle must be distance-regular, or equivalently since the point-collinearity graph has diameter $2$, that the graph is \textit{strongly regular} with parameters $\left((s+1)(st+1),s(t+1),s-1,t+1\right)$\footnote{These parameters can be interpreted geometrically as: the total number of points, the number of points collinear (but not equal to) a point $p$, the number of common neighbours to a pair of collinear points $p$ and $q$ and the number of common neighbours when $p$ and $q$ are not collinear.}.
 
An automorphism of an incidence geometry is a pair $(\rho,\sigma)$ of permutations on the point and line sets respectively, which together preserve incidence. The automorphism group of the geometry is the set of all of its automorphisms. The automorphism group of a generalised quadrangle may be identified with the set of all permutations on the points which preserve collinearity, because the structure of quadrangles guarantees that every permutation on points preserving collinearity determines a unique permutation on the lines which preserves incidence.

In specific cases there is much that can be said about the automorphisms and automorphism groups of specific quadrangles. However, our work requires results that apply in complete generality, and there are rather fewer of those. The following results are the ones that we will need.

\begin{lemma}[\cite{Buekenhout:1993bf}] \label{OOE}
Let $\mathcal{Q}$ have order $(s,t)$, and $\theta$ be an automorphism of prime order $x$ greater than both $s+1$ and $t+1$. Then $\theta$ cannot fix any points, so $x$ cannot divide the size of the point-stabilisers, and must divide $\abs{\mathcal{P}}$, and dually for lines.
\end{lemma}
%

\begin{lemma}\label{LOL}
Let $G$ be an automorphism group of a generalised quadrangle $\mathcal{Q}$ of order $(s,t)$, and suppose $G$ acts transitively on $\mathcal{P}$. Let $L^{G}$ be an orbit of $G$ on $\mathcal{L}$. Then there is a number $1 \leq k \leq t+1$ such that each point is incident with exactly $k$ lines from $L^{G}$. Consequently, $L^{G}$ has $k(st+1)$ lines in it. If $G$ acts primitively on $\mathcal{P}$ then $k \neq 1,t$.
\end{lemma}

\begin{proof}
We simply count the pairs $(p,M) \in \mathcal{P} \times L^{G}$, such that $p$ and $M$ are incident, in two ways. There are $|L^{G}|$ lines, each incident with $s+1$ points, giving $|L^{G}|(s+1)$ pairs. Since $G$ is transitive on points, for all $p, q \in \mathcal{P}$, there is some $g \in G$ which maps $p$ to $q$, and all lines of $L^{G}$ at $p$ to lines of $L^{G}$ at $q$. This means that every point is incident with the same number, $k$, of lines of $L^{G}$. The number of pairs is thus also $k \abs{\mathcal{P}}$. Combining gives $|L^{G}|=k\frac{\mathcal{\abs{P}}}{s+1}=k(st+1)$. 

If $k=1$, every point is incident with one line of $L^{G}$, so the lines of $L^{G}$ determine a partition of the point set which is preserved by $G$, i.e., they form a system of imprimitivity, which is not possible if $G$ acts primitively on points. This also applies for $k=t$, for the complement of a line orbit with $k=t$ is a line orbit with $k=1$. 
\end{proof}

As far as we can tell, after surveying the literature and asking several experts in this field, Lemma \ref{LOL} has not been utilised in the study of generalised polygons before, although it is essentially just the observation that the point-orbits and line-orbits form a tactical configuration (see \cite[Chapter 1]{Dembowski:1997sy} for more), together with the fact that there is just one point-orbit in this case.

\section{The algorithm} \label{algorithm}

Let $G$ be an arbitrary finite group. The goal is to determine whether there are any generalised quadrangles upon which $G$ can act point-primitively (and faithfully). In other words, are there any generalised quadrangles where $G$ appears as a subgroup of the full automorphism group, and is primitive on the point set in this action? Prior work has focused on showing that certain groups or types of groups cannot do this. Showing definitively that a group can do it is much harder, because essentially there is no way to show whether a group can act point-primitively on a quadrangle other than to search for a quadrangle such that it does. 

The process discussed here works because the fact that $G$ acts primitively on the points of $\mathcal{Q}$ implies certain restrictions on the parameters of $\mathcal{Q}$. Hence you only need to check a much reduced subset of all quadrangles  in order to find those acted upon point-primitively by $G$. By doing this repeatedly and by using several different properties that any quadrangle acted on point-primitively by $G$ must have, one can whittle down the set of quadrangles that must be checked directly to a (hopefully quite small) finite number. As it happens, the number that must be checked often turns out to be zero under our method. The algorithm that we have developed uses the following steps:
\begin{enumerate}
\item 
The stabilisers of points in this primitive action of $G$ must form a conjugacy class of maximal subgroups of $G$. Thus, any hypothetical quadrangle which is acted on point-primitively by $G$ must have a number of points equal to the index of some maximal subgroup of $G$ (by the Orbit-Stabiliser theorem).

\item Each generalised quadrangle has an order $(s,t)$ which obeys $\abs{\mathcal{P}}=(s+1)(st+1)$, $s \leq t^{2}$, $t\leq s^{2}$ and $s+t$ divides $st(st+1)$ (see Lemma \ref{basiclemma}). The first of these implies that $s+1$ must divide $\abs{\mathcal{P}}$. So we set $s+1=j$, run $j$ over all of the divisors of each possible $\abs{\mathcal{P}}$ found in the last step and solve for \[t = \frac{\abs{\mathcal{P}}/j-1}{j-1}.\] It is then checked whether $t$ is an integer and obeys the other restrictions. This generates a list of possible $(s,t)$'s, but not all will be compatible with $G$ acting point-primitively on these hypothetical quadrangles. This is where Lemmas \ref{OOE} and \ref{LOL} come in. 
The use of Lemma \ref{OOE} is clear: find the set of prime divisors of $\abs{G}$, take those greater than $s+1$ and $t+1$ and test whether any divide the sizes of the point-stabilisers (which are known) or do not divide $\abs{\mathcal{L}}$ (the sizes of the line stabilisers are not known). If either of these occur, then $(s,t)$ can be eliminated, for it is incompatible with $G$\footnote{This does not tend to be a very effective test, for the size of $s$ and $t$ often exceeds all prime divisors of $\abs{G}$.}.
That the sizes of the line stabilisers are not known is partially negated by Lemma \ref{LOL}. The size of each line orbit is the index in $G$ of the stabiliser of a line in that orbit. Lemma \ref{LOL} tells us that this index must be divisible by $st+1$. The line-orbits partition the line set, so there must be some set of subgroups of $G$ whose indices are divisible by $st+1$ and sum to the number of lines. This is a surprisingly difficult test to pass, and it eliminates the vast majority of all cases which reach this step.

\item The algorithm now has a set of possibilities for $\abs{\mathcal{P}}$ and has split those by the possible $(s,t)$'s. These are further subdivided and reduced by looking at the possibilities for the sets of points, $\mathcal{N}_{p}$, which can be collinear with a point $p$. Fixing $p$ in a hypothetical quadrangle of order $(s,t)$, there will be $s(t+1)$ other points collinear with $p$. The stabiliser $G_{p}$ must preserve this set, so $\mathcal{N}_{p}$ is the union of a some orbits of $G_{p}$. The action is determined by specifying $G$ and $G_{p}$ and hence these orbits (and their sizes, known as \textit{subdegrees}) can be calculated. The sets of orbits whose combined size is $s(t+1)$ can then found, and these are the possibilities for the set $\mathcal{N}_{p}$. As of yet, we know of no simple test which can be used to eliminate some of these orbit combinations as being valid neighbourhood possibilities. This is currently the most significant weakness of this process.
\end{enumerate}

At this point enough information has been collected that each remaining case (consisting of parameters $\abs{\mathcal{P}}$, $G_{p}$, $(s,t)$, $\mathcal{N}_{p}$) contains at most one generalised quadrangle. This is because specifying $G_{p}$ (which determines the action on points) and the set of points collinear with a given point $p$ suffices to determine the collinearity of any two points.
Hence each case is associated to a unique point-collinearity graph, which by construction is acted on point-primitively by $G$. This graph need not be associated to a generalised quadrangle however. First of all, as mentioned in Section \ref{2}, the point-collinearity graph of a quadrangle must be strongly regular with parameters $\left((s+1)(st+1),s(t+1),s-1,t+1\right)$. 
The graphs coming out of each case are checked to see whether this is true or not.

Even if a graph passes this test, it need not be the point-collinearity graph of any generalised quadrangle. If it is not, it is called a \textit{pseudo-geometric graph}, and these are quite rare. There are also computational methods (e.g. using \cite{GRAPE}) for determining whether a given graph is associated to any quadrangle or not, and finding such quadrangles. It is also known that there will be a unique such quadrangle, if it exists. By construction the quadrangles found have $G$ as a point-primitive automorphism group, and are (up to isomorphism) the only such quadrangles\footnote{Hence, the process does not construct a list of quadrangles and check if they are acted on primtively by $G$, as suggested in the informal motivation at the beginning of this section, but instead constructs objects with $G$ as a point-primitive automorphism group and checks if they are quadrangles.}.

\section{Analysing the almost simple sporadic groups}

The main component of the proof of Theorem \ref{maintheorem} is the automation of the process by way of a computer program. This was done using the \textsf{GAP} \cite{GAP4} computer algebra software, with assistance from the \AtlasRep\ \cite{AtlasRep1.5.1}, \textsf{GRAPE} \cite{GRAPE}, \textsf{Design} \cite{Design} and \textsf{FinInG} \cite{fining} packages. The code for this implementation can be found in Appendix \ref{appendix:code}.

A group $G$ is almost simple if $T \cong \text{Inn}(T) \leq G \leq \text{Aut}(T)$ for some finite non-abelian simple group $T$. The automorphism groups of all classical and dual classical polygons are almost simple of Lie type ($T$ is a Lie type simple group), so the other almost simple groups are a natural place to place to start searching for groups which could be the point-primitive automorphism group of a quadrangle. Focusing specifically on the sporadic groups provides a finite but challenging set of examples, and \AtlasRep\ conveniently provides access to good representations of these groups. This means that the correctness of the stated theorem relies upon the correctness of the contents of \AtlasRep\footnote{But not its completeness. At the time of writing \AtlasRep\ had gaps, how they were overcome is discussed below.}.

\subsection{Verification of the implementation}

The code was tested by applying it to the simple groups associated to the classical families for small $q$. These are known to act point-primitively on their respective quadrangles \cite{Van-Maldeghem:1998sj}. Table \ref{known} displays the groups checked and the results. All of the known quadrangles for these groups were found, giving confidence that our implementation won't miss any quadrangles should they exist (which is by far more important, and more likely to be a problem, than spuriously detecting non-existent quadrangles). This demonstrates one possible use of implementing the program like this: to directly, quickly and independently check the results about which groups act primitively on quadrangles that have been obtained by other means.

\begin{table}[!ht]
\centering
\small
\begin{tabular}{ l l  c  c  l }
\toprule
$G$ & $M$ & $(s,t)$ & Graph? & $\mathcal{Q}$ \\
\midrule

\rowcolor{green!40!yellow!25}$\PSp(4,2)^{\prime} \cong A_{6}$ & $S_{4}$ & $(2,2)$ & $\checkmark$ & $W(3,2) \cong Q(4,2)$\\
\rowcolor{green!40!yellow!25} & $S_{4}$ & $(2,2)$ & \checkmark & $W(3,2) \cong Q(4,2)$\\

$\PSp(4,3) \cong \PSU(4,2)$ & $2^{4}:A_{5}$ & $(2,4)$ & $\checkmark$ & $Q^{-}(5,2)$\\
 & $3^{1+2}.2A_{4}$ & $(3,3)$ & $\checkmark$ & $W(3,3)$\\
 & $3^{3}.S_{4}$ & $(3,3)$ & $\checkmark$ & $Q(4,3)$\\
 & $2.(A_{4}\times A_{4}).2$ & $(4,2)$ & $\checkmark$ & $H(3,4)$\\

\rowcolor{green!40!yellow!25} $\PSp(4,4)$ & $2^{6}.(3 \times A_{5})$ & $(4,4)$ & $\checkmark$ & $W(3,4) \cong Q(4,4)$\\
\rowcolor{green!40!yellow!25} & $2^{6}.(3 \times A_{5})$ & $(4,4)$ & $\checkmark$ & $W(3,4) \cong Q(4,4)$\\
\rowcolor{green!40!yellow!25} & $S_{6}$ & $(9,15)$ & $\times$ & - \\

$\PSp(4,5)$ & $5^{1+2}:4A_{5}$ & $(5,5)$ & $\checkmark$ & $W(3.5)$\\
 & $5^{3}:(2 \times A_{5}).2$ & $(5,5)$ & $\checkmark$ & $Q(4,5)$\\
 & $2.(A_{5} \times A_{5}).2$ & $(4,16)$ & $\times$ & - \\

\rowcolor{green!40!yellow!25} $\PSU(4,3)$ & $3^{4}:A_{6}$ & $(3,9)$ & \checkmark & $Q^{-}(5,3)$\\
\rowcolor{green!40!yellow!25} & $3^{1+4}:(2.S_{4})$ & $(9,3)$ & \checkmark & $H(3,9)$\\
\rowcolor{green!40!yellow!25} & $\PSU(3,3)$ & $(11,4)$ & $\times$ & -\\

$\PSU(5,2)$ & $2^{1+6}:(3^{2}:3:Q_{8})$ & $(4,8)$ & $\checkmark$ & $H(4,4)$\\
 & $2^{4+4}:GL(2,4)$ & $(8,4)$ & $\checkmark$ & $H(4,4)^{D}$\\
 & $(3^{2}:3:Q_{8}):3 \times S_{3}$ & $(9,39)$ & $\times$ & - \\

\bottomrule
\end{tabular}
\caption{This table shows the group, the maximal subgroups which had $(s,t)$'s, if the associated graph was strongly regular, and which generalised quadrangle it represented (if any).}
\label{known}
\end{table}

\subsection{Results from testing the almost simple sporadic groups}

Table \ref{Sporadic Results} shows a list of all of the almost simple sporadic groups and contains information about the operation and output of our program when applied to each of them in turn. No generalised quadrangle was found in any of the cases.  It must be noted that our automated computer code did not suffice to analyse all of these groups, so in some cases manual work was required in order to complete the process. The problems that were encountered and their solutions are discussed below. That information should also clear up some of the odd features of Table \ref{Sporadic Results} that the reader may notice. Unless a group is mentioned below, the computer program given in Appendix \ref{appendix:code} was able to complete the required analysis without issue.

\begin{table}[h!]
\centering
\tiny
\begin{tabular}{ l l  c  c  c  c  c }
\toprule
$G$ & $M$ & $(s,t)$ & $(s,t)^{*}$ & Subdegrees & NC & $\mathcal{Q}$ \\
\midrule

\rowcolor{green!40!yellow!25}$M_{11}$ & $2.S_{4}$ & $(4,8)$ & $(4,8)$ & $[8,12,24^{(4)},48]$ & $[0,1,1,0]$ & -\\

$M_{12}$ & - & - & - & - & - & -\\
$M_{12}.2$ & - & - & - & - & - & -\\

\rowcolor{green!40!yellow!25}$M_{22}$ & - & - & - & - & - & -\\
\rowcolor{green!40!yellow!25}$M_{22}.2$ & - & - & - & - & - & -\\

$M_{23}$ & - & - & - & - & - & -\\

\rowcolor{green!40!yellow!25}$M_{24}$ & - & - & - & - & - & -\\

$J_{1}$ & $F_{42}$ & $(21,9)$ & - & - & - & -\\

\rowcolor{green!40!yellow!25}$J_{2}$ & $3.A_{6}.2$ & $(9,3)$ & - & - & - & -\\
\rowcolor{green!40!yellow!25}  & $5^{2}:D_{12}$ & $(13,11)$ & - & - & - & -\\
\rowcolor{green!40!yellow!25}$J_{2}.2$ & $3.A_{6}.2^{2}$ & $(9,3)$ & $(9,3)$ & $[36,108,135]$ & [1,0,0] & -\\
\rowcolor{green!40!yellow!25}  & $5^{2}:(4 \times S_{3})$ & $(13,11)$ & $(13,11)$ & $[15,25,50^{(2)},75^{(5)},150^{(6)},300^{(2)}]$ & - & -\\

$J_{3}$ & $2^{2+4}:(3\times S_{3})$ & $(44,22)$ & - & - & - & -\\
$J_{3}.2$ & $2^{2+4}:(S_{3}\times S_{3})$ & $(44,22)$ & - & - & - & -\\

\rowcolor{green!40!yellow!25}$J_{4}$ & - & - & - & - & - & -\\

$Co_{1}$ & - & - & - & - & - & -\\

\rowcolor{green!40!yellow!25}$Co_{2}$ & $M_{23}$ & $(161,159)$ & - & - & - & -\\

$Co_{3}$ & - & - & - & - & - & -\\

\rowcolor{green!40!yellow!25}$Fi_{22}$ & $O_{7}(3)$ & $(39,9)$ & - & - & - & -\\
\rowcolor{green!40!yellow!25}  & $O_{7}(3)$ & $(39,9)$ & - & - & - & -\\
\rowcolor{green!40!yellow!25}  & $O_{8}^{+}(2):S_{3}$ & $(25,95)$ & - & - & - & -\\
\rowcolor{green!40!yellow!25}  & & $(39,49)$ & (39,49) & [1575,22400,37800] & - & -\\
\rowcolor{green!40!yellow!25}$Fi_{22}.2$ & $O_{8}^{+}(2):S_{3} \times 2$ & $(25,95)$ & - & - & - & -\\
\rowcolor{green!40!yellow!25} & & $(39,49)$ & - & - & - & -\\

$Fi_{23}$ & $[3^{10}].(L_{3}(3)\times 2)$ & $(2991,689)$ & - & - & - & -\\

\rowcolor{green!40!yellow!25}$Fi_{24}^{\prime}$ & $Fi_{23}$ & $(115,23)$ & - & - & - & -\\
\rowcolor{green!40!yellow!25}$Fi_{24}^{\prime}.2$ & $Fi_{23}\times 2$ & $(115,23)$ & (115,23) & $[31671,275264]$ & - & -\\

$HS$ & - & - & - & - & - & -\\
$HS.2$ & - & - & - & - & - & -\\

\rowcolor{green!40!yellow!25}$McL$ & $M_{22}$ & $(8,28)$ & - & - & - & -\\
\rowcolor{green!40!yellow!25} & $M_{22}$ & $(8,28)$ & - & - & - & -\\
\rowcolor{green!40!yellow!25}$McL.2$ & - & - & - & - & - & -\\

$He$ & - & - & - & - & - & -\\
$He.2$ & - & - & - & - & - & -\\

\rowcolor{green!40!yellow!25}$Ru$ & ${}^{2}F_{4}(2)$ & $(9,45)$ & - & - & - & -\\
\rowcolor{green!40!yellow!25} & $2^{6}.U_{3}(3).2$ & $(57,57)$ & $(57,57)$ & $[63,756,2016^{(3)},16128^{(2)}\ldots]$ & - & -\\

$Suz$ & $U_{5}(2)$ & $(41,19)$ & - & - & - & -\\
& $3^{2+4}:2.(A_{4}\times 2^{2}).2$ & $(129,191)$ & - & - & - & -\\
$Suz.2$ & $U_{5}(2):2$ & $(41,19)$ & - & - & - & -\\
& $3^{2+4}:2.(S_{4}\times D_{8})$ & $(129,191)$ & (129,191) & $[72,144,486^{(2)},729,1944,3888^{(2)},4374,$ & & \\

& & & & $7776^{(3)},8748,11664,17496,23328\ldots]$ & - & -\\

\rowcolor{green!40!yellow!25}$O'N$ & $L_{3}(7):2$ & $(19,323)$ & - & - & - & -\\
\rowcolor{green!40!yellow!25} & $L_{3}(7):2$ & $(19,323)$ & - & - & - & -\\
\rowcolor{green!40!yellow!25}$O'N.2$ & - & - & - & - & - & -\\

$HN$ & $A_{12}$ & $(149,51)$ & - & - & - & -\\
$HN.2$ & $S_{12}$ & $(149,51)$ & $(149,51)$ & $[1,462,5040,10395\ldots]$ & - & -\\

\rowcolor{green!40!yellow!25}$Ly$ & - & - & - & - & - & -\\

$Th$ & - & - & - & - & - & -\\

\rowcolor{green!40!yellow!25}$B$ & - & - & - & - & - & -\\

$M$ & - & - & - & - & - & -\\

\bottomrule
\end{tabular}
\captionsetup{justification=centering} 
\caption{$(s,t)^{*}$: Lists $(s,t)$ possibilities which pass the tests of Lemmas \ref{OOE} and \ref{LOL}.\\
Subdegrees: Lists the subdegrees of the relevant action, with multiplicity.\\
 NC: Gives the number of different combinations for the point neighbourhoods.\\
$\mathcal{Q}$: Lists the quadrangles output by the program.}
\label{Sporadic Results}
\end{table}

\begin{itemize}
\item $Co_{1}$: Not all maximal subgroups of $Co_{1}$ are accessible via \AtlasRep. However, the online website for \AtlasRep\ contains a list of the orders and indices of these subgroups (which is all of the information required at first). This information was manually entered into the part of the program which finds the initial possible $(s,t)$ values. No possibilities were found for any of the maximal subgroup classes. Hence $Co_{1}$ was eliminated. 
\item $Fi_{22}.2$: Similarly, \AtlasRep\ misses many maximal subgroups of $Fi_{22}.2$. Following the same process as for $Co_{1}$ finds a single possibility: $(s,t)=(25,95)$ for the maximal subgroup classes $O_{8}^{+}(2):S_{3} \times 2$. This possibility is not eliminated by the use of Lemmas \ref{OOE} or \ref{LOL}. This phenomenon, of the simple group being eliminated without issue, but the almost simple group requiring more work occurs several times\footnote{The problem is mainly in how the test based on Lemma \ref{LOL} is conducted. Finding all subgroup indices directly is infeasible, but they must all be multiples of the index of some maximal subgroup. Hence, line stabilisers must have indices which are multiple of lcm$(\abs{G:M},st+1)$ for some maximal subgroup $M$. This coarser information often suffices, because the indices of maximal subgroups of simple groups tend to be large and the lcm is often larger than the number of lines. But in the almost simple case the index $2$ maximal subgroup causes problems. Often the problem can be resolved without resorting to the trick discussed for $Fi_{22}.2$, and if so this is shown.}.

This problem can be resolved as follows. Suppose $G$ was an almost simple group with corresponding simple group $T$, and that $M\leqslant G$ was maximal. Often, $T \cap M$ will be maximal in $T$\footnote{If $T\cap M$ isn't maximal, $M$ is called a \emph{novelty} maximal subgroup of $G$. If $T \cap M$ is maximal, $M$ is a non-novelty maximal subgroup of $G$.}. Suppose this was the case, and that $G$ acted point-primitively on some quadrangle with stabiliser $M$. Then $T$ would act on it with point-stabiliser $T \cap M$ which is maximal in $T$. Hence the action of $T$ would also be primitive with stabiliser $T \cap M$. But, if $T$ has already been analysed and eliminated, this cannot occur, so we rule out $G$ acting with stabiliser $M$.

For this case, $O_{8}^{+}(2):S_{3} \times 2$ is a non-novelty maximal subgroup \cite{ATLAS}, the intersection with $Fi_{22}$ is the maximal subgroup $O_{8}^{+}(2):S_{3}$. Hence this case, and $Fi_{22}.2$ are eliminated.

\item $Fi_{23}$: Similar to $Co_{1}$, $Fi_{23}$ does not have all of its maximal subgroups in \AtlasRep. Manually checking the $(s,t)$'s finds the one possibility $(2991,689)$. Manually entering this into the test based on Lemma \ref{LOL} eliminates this possibility as well.

\item $Fi_{24}^{\prime}$: Manual $(s,t)$ checking reveals $(s,t)=(115,23)$, with the $Fi_{23}$ class as the only possibility. This is eliminated by using Lemma \ref{LOL}. This time, the maximal subgroup indices are not in \AtlasRep. Instead, a list of the structure descriptions of the maximal subgroups was sourced from \cite{MaximalSubgroups} and it is simple, if tedious, to calculate the indices using these. The indices were then run through the $(s,t)$ finding step.

\item $Fi_{24}^{\prime}.2$ : Similar to $Fi_{24}^{\prime}$. Again, the only possibility found is $(s,t)=(115,23)$ associated to maximal subgroups $Fi_{23} \times 2$. This is a non-novelty maximal subgroup \cite{ATLAS}, so this case is eliminated also.

\item $HN.2$: This group has all of its maximal subgroups available in \AtlasRep, so things initially proceed as normal. The one $(s,t)$ possibility, $(149,51)$, is not eliminated, and so the program proceeds to the next step of trying to calculate the subdegrees of the required action. $HN.2$ proves to be too large for this to work. The relevant maximal subgroup is non-novelty, so it can be eliminated that way. However, in this case The Atlas \cite{ATLAS} lists the subdegrees of the relevant action, and manual inspection shows that no combination sums to $s(t+1)$.

\item $B$: The maximal subgroups are in \AtlasRep, however $B$ is far too large for practical computation. Instead, a list of the maximal subgroup indices are extracted and manually run through the $(s,t)$ finder step, producing no possibilities.

\item $M$: Same as for $B$, however this time \textsf{GAP}'s factorisation function proved inadequate. Instead a version of the $(s,t)$ finder was written in \textsf{Python3} and used instead. No $(s,t)$'s were found.
\end{itemize}
\vspace{-0.2cm}
This concludes the proof of Theorem \ref{maintheorem}.

\section{The Hemisystem Conjecture}

As useful as the above discussed process is, it has its limitations. The first few steps, using maximal subgroup indices and finding valid $(s,t)$ values, are simple and cheap, especially since a lot is known about maximal subgroups of simple groups. Beyond this first step, the method begins to rely on heavy computations with the relevant group which can often overwhelm basic computers. Manual intervention (as in the last part of the previous section) can only get you so far. This limits the possibility of smoothly analysing large numbers of groups. 

Our efforts to improve the program have focused on the first few steps, so as to eliminate as many cases as possible and avoid the intensive group calculations whenever we can. The tests on the $(s,t)$ values which are based on Lemmas \ref{OOE} and \ref{LOL} were discovered as a result of this focus. They are very successful, together eliminating the vast majority of cases. Without Lemma \ref{LOL} we would have been unable to rule out $J_{1}$ or $Fi_{23}$. Despite the utility of these lemmas, we continue to seek other tests that can be added to the process.

One possibility relies on the following observation. The way that Lemma \ref{LOL} is used relies on the fact that the size of the line-orbits must be divisible by $st+1$, which only requires the relevant automorphism group to be point-transitive. The added condition for when the group is point-primitive is minor and not currently useful. This means that our use of Lemma \ref{LOL} is not using the full strength of the assumption that the group is acting point-primitively. We suspected that a much stronger, and more useful, restriction on possible line-orbits could be proven if the point-primitivity condition was utilised more completely. Our investigations into this possibility led us to the following conjecture:

\begin{conjecture*}[The Hemisystem Conjecture]
Suppose that $G$ is a group acting point-primitively on a generalised quadrangle $\mathcal{Q}$. Then $G$ either acts transitively on the lines, or has exactly two line-orbits, both of equal size.
\end{conjecture*}

When the group $G$ has exactly two line-orbits of equal size, we say that $G$ \textit{has a hemisystem}. The reason for this should become clear later. The inspiration for this conjecture comes from two sources. The first are the examples of quadrangles with point-primitive groups: in all known cases the statement of the conjecture holds. The second is that the theory of objects called $k$-covers lends some heuristic support to the idea behind the conjecture.

\subsection{The known examples}

The known generalised quadrangles with point-primitive automorphism groups are the classical (and dual-classical) quadrangles as well as $\GQ(3,5)$ and $\LSce$. We sought to examine the line-orbits of the point-primitive groups on these quadrangles. Our efforts to do this were greatly helped by the classification of point-transitive automorphism groups of classical quadrangles assembled in \cite{GiudiciTransitiveSubgroups}. Some of the information from this list is reproduced in Table \ref{GiudiciAudit}. The table describes the point-transitive subgroups up to conjugacy in the full automorphism group. Since primitive implies transitive, all of the point-primitive groups appear in this list. 
\begin{table}[!ht]
\centering
\scriptsize
\begin{tabular}{ l p{6cm} l l}
\toprule
$GQ$&Transitive Subgroup&Conditions& Comment\\
\midrule
$W(3,q)$&contains $\PSp(4,q) $& & primitive\\
&contains $\PSL(2,q^{2}) $& & stabilises spread (\cite[Tables 8.12, 8.14]{Bray:2013aa})\\
&$2^{4}.A_{5},2^{4}.5,S_{5},2^{4}.D_{10},2^{4}.F_{20},2\times S_{5},2^{4}.S_{5}$&$q=3$ &imprimitive\\
&$A_{5},S_{5}$&$q=2$ & stabilises spread\\
\midrule
$Q(4,q)$, $q$ odd&contains $\PSp(4,q) $& &primitive\\
&$2^{4}.F_{20},2^{4}.A_{5},2^{4}.S_{5}$&$q=3$ &imprimitive\\
\midrule
$H(3,q^{2})$&contains $\PSU(4,q)$&&primitive\\
&contains $\PSL(3,4)$&$q=3$&primitive\\
\midrule
$Q^{-}(5,q)$&contains $\PSU(4,q)$&&primitive\\
&contains $\SU(3,q)$&&stabilises spread (\cite[Table 8.33]{Bray:2013aa})\\
&$C_4\times \PSL(2,7)$, $(C_4\times \PSL(2,7)).2$ &$q=3$&stabilises spread\\
& $C_2 \times (\PSL(2,7):2)$, $(C_2 \times (\PSL(2,7):2)).2$&$q=3$&stabilises spread\\
&$A_{7}$&$q=5$&`exceptional'\\
&$C_{513}:C_{9},(C_{513}:C_{9}):C_{2}$&$q=8$&`exceptional'\\
&Contains a regular extraspecial $3$-group&$q=2$&stabilises spread\\
&contains $\PSL(3,4):2$&$q=3$&imprimitive\\
\midrule
$H(4,q^{2})$&contains $\PSU(5,q)$&&primitive\\
\midrule
$H(4,q^{2})^{D}$&contains $\PSU(5,q)$&&primitive\\
\bottomrule 
\end{tabular}
\caption{} 
\label{GiudiciAudit}
\end{table}
Our first task was to locate the point-primitive groups. For sufficiently large $q$, the table shows that all of the groups fall into orderly infinite families. There is always the simple group part of the full automorphism group, which is known to always be point and line primitive. The other infinite families of point-transitive subgroups are known to arise as the stabilisers of geometric objects in the ambient projective spaces that the quadrangles live in, so their actions are well understood. In fact, they are subgroups of the stabilisers of $1$-covers (a.k.a, \emph{spreads}: sets of lines with one line incident at each point), so the second part of Lemma \ref{LOL} implies the groups are not point-primitive. These spread stabilisers happen to be maximal in their associated collineation groups (see the indicated table in \cite{Bray:2013aa}, \cite{Cameron:1997aa}, or the relevant chapter of \cite{KleidmanLiebeck}), so the only primitive overgroup is the whole group, and so all cases are consistent with the conjecture.

This just leaves the exceptional point-transitive groups of the small $q$ quadrangles. These were analysed using \textsf{GAP} together with the \textsf{FinInG} \cite{fining} and \textsf{Design} \cite{Design} packages which were used to construct the necessary quadrangles and automorphism groups. We then calculated every conjugacy class of subgroups and determined if they were transitive or primitive on the points. Whether or not they were primitive\footnote{So that we would be able to see the line-orbits of primitive and imprimitive groups and make a comparison.}, if they were transitive we then looked at the line-orbits of (a representative of) each class. This analysis could be carried out except for $Q^{-}(5,5)$ and $Q^{-}(5,8)$, due to the large size of these examples.
These were dealt with by only finding the exceptional subgroup we were required to check, rather than searching all subgroup classes. This can be done by utilising other known constructions of these subgroups. For $Q^{-}(5,5)$, the $A_{7}$ subgroup class that we need to check can be found by using the projective representation of $A_{7}$ inside $\PGL(6,5)$ provided by \AtlasRep. For $Q^{-}(5,8)$, the relevant subgroups are the centralisers of Sylow-19 subgroups. Checking all relevant overgroups\footnote{The groups between these groups and their normaliser, which is a maximal subgroup. See the appendices of \cite{Bray:2013aa}.} finds no point-primitive examples except those already accounted for.

The results are as follows. All of the point-primitive groups that we examined were also line-transitive, except for two classes which were related to $\PSL(3,4) \leq \Aut(H(3,9))$, which each had two line-orbits, both of which contained half of the lines. So, for classical examples at least, it is almost true that point-primitive implies line-transitive. In fact, this was true for all examples we could think of off the top of our heads, which inspired us to carry out this survey at all. On the basis of the one exception, we made the conjecture as stated above.

However, it was possible that this pattern that we spotted was just an artefact of the classical quadrangles. Thus, it was necessary to analyse the remaining two quadrangles $\GQ(3,5)$ and $\LSce$. Both have symmetry groups which are primitive of Holomorph Affine type. $\GQ(3,5)$'s automorphism group was analysed by finding and checking every conjugacy class of subgroups (as done for the classical quadrangles), however $\LSce$'s group was simply too big for this to work.
A more efficient way to find primitive subgroups of the symmetry group of $\LSce$ is to use the result of \cite{PraegerInclusions}
on inclusions of primitive groups in other primitive groups.

\begin{theorem}[{\cite[Proposition 5.2]{PraegerInclusions}}]
Suppose that $G$ is a finite primitive permutation group of Holomorph Affine type with translation subgroup $T$, and that $H \leq G$ is also primitive. Then either $H$ is of Holomorph Affine type and contains $T$ or:
\begin{itemize}
\item $H$ is of Almost Simple type and is isomorphic to $\PSL(2,7)$, and $G=\mathrm{AGL}(3,2)$ or
\item $H$ is of Product Action type and $H$ has primitive component $\PSL(2,7)$.
\end{itemize}
\end{theorem}

Note that $\abs{\PSL(2,7)} =168$ is divisible by $7$. But the order of the symmetry group $G$ of $\LSce$ is not divisible by $7$. Thus $\PSL(2,7)$ is not a subgroup of $G$, and thus any primitive subgroup of $G$ is of Holomorph Affine type and still contains the translation subgroup $T$.

The correspondence theorem from group theory tells us that there is a bijective correspondence between conjugacy classes of subgroups in $G$ which contain $T$ and those in $G/T$. Because $G/T$ is much smaller than $G$, its conjugacy classes of subgroups are more easily found. \textsf{GAP} finds $86$ classes of subgroups of $G/T$ which are then pulled back to subgroups of $G$ containing $T$. These subgroups will all be point-transitive, for they all contain $T$, and the proposition tells us that every class of primitive subgroups is represented in the list. Hence, we obtained all of the point-primitive subgroups, and a number of extra transitive subgroups for comparison. The results are as follows:

\begin{table}[!ht]
\small
\begin{tabular}{lp{4cm}p{4cm}l}
$\mathcal{Q}$ & $\#$ classes of primitive subgroups of $\Aut(\mathcal{Q})$ & $\#$ having two line-orbits (a hemisystem) & $\#$ line-transitive\\
\midrule
$\GQ(3,5)$ & 10 & 4 & 6\\
$\LSce$ & 11 & 4 & 7\\
\bottomrule
\end{tabular}
\end{table}

Thus, we have shown that all known quadrangles agree with the hemisystem conjecture. The evidence from these non-classical examples is especially compelling. Despite these two quadrangles being very different from the classical examples, the same result holds. There is also some evidence in the results from the point-imprimitive groups. As one can see by looking at the tables of line orbit sizes contained in Appendix \ref{appendix:lineorbits}, the number of line-orbits for imprimitive groups are quite varied. But the primitive examples stand out as having
only one or two orbits on lines. This might just be an artefact of the primitive groups being larger and so having fewer line-orbits, but to our eyes there does seem to be something sharply different between large but imprimitive subgroups and primitive ones.

\subsection{Hemisystems} 


What Lemma \ref{LOL} tells us is that the line-orbits of a point-transitive automorphism group form geometric objects called $k$-covers of lines in generalised quadrangles. These are well known objects, appearing more commonly under the guise of the dual notion of \textit{m-ovoids}. Formally, if $\mathcal{G}$ is some generalised quadrangle, an $m$-ovoid of $\mathcal{G}$ is a subset $\mathcal{O} \subset \mathcal{P}$ such that every line of the quadrangle is incident with exactly $m$ points of $\mathcal{O}$. A $k$-cover is a subset $\mathcal{K}$ of $\mathcal{L}$ such that every point is incident with exactly $k$ lines of $\mathcal{K}$.

A recurring theme in the literature is that the most important examples of these objects are the \textit{hemisystems}, $m$-ovoids where $m=\frac{s+1}{2}$ (containing half of the points) and dually $k$-covers where $k = \frac{t+1}{2}$. An important early result in their study was B. Segre's 1965 proof \cite{SegreHemisystems} that any nontrivial $k$-cover in a $H(3,q^{2})$ quadrangle is a hemisystem. This was inspired by the existence of hemisystems for $H(3,9)$ stabilised by the groups $\PSL(3,4) \leq \Aut(H(3,9))$.

This is not the place to discuss all of the results about $k$-covers, hemisystems and their relation to the general theory of generalised quadrangles, so we limit ourselves to one example which exemplifies the sort of pattern which helped to inspire the conjecture. This example comes from Thas' 1989 paper \cite{THAS1989103}. This paper introduces $m$-ovoids and $k$-covers in full generality for the first time and proves many of the basic results about them. The paper deals with $m$-ovoids, but its results are easily dualised to ones about $k$-covers.  It assumes $m$-ovoids are non-trivial ($m\neq 0,s+1$). The notation $\mathcal{G}(A)$ denotes the induced subgraph of the point-collinearity graph $\mathcal{G}(\mathcal{P})$ on vertex set $A \subset \mathcal{P}$: the vertices are the elements of $A$, and two vertices are adjacent if and only if they correspond to distinct collinear points.
Thas proves the following generalisation of Segre's result:

\begin{corollary}[\cite{THAS1989103}] \label{Corollary}
If the generalised quadrangle $\mathcal{G}$ has order $(s,s^{2})$ and contains a non-trivial $m$-ovoid $\mathcal{O}$ then $m=\frac{s+1}{2}$. If $\mathcal{O}$ is an $\frac{s+1}{2}$-ovoid of $\mathcal{G}$ then $\mathcal{G}(\mathcal{O})$ and $\mathcal{G}(\mathcal{P}\setminus \mathcal{O})$ are strongly regular with parameters $v=\frac{s+1}{2}(s^{3}+1)$, $k=\frac{s-1}{2}(s^{2}+1)$, $\lambda=\frac{s-3}{2}$, $\mu=\frac{(s-1)^{2}}{2}$.
\end{corollary}

However the crucial result which supports the Hemisystem Conjecture is the following:

\begin{theorem}[\cite{THAS1989103}] \label{placeholder}
Let $\mathcal{O}$ be an $m$-ovoid of a generalised quadrangle $\mathcal{G}$ of order $(s,t)$. If $\mathcal{G}(\mathcal{O})$ is strongly regular then one of the following cases occurs:
\begin{itemize}
\item $m=\frac{s+1}{2}$ and $t=s^{2}$.
\item $m< \frac{s+1}{2}$. If m=1 then $t\leq s^{2}-s$. If $m>1$ then $t \leq s^{2}-2s$.
\item $m>\frac{s+1}{2}$ and $t=s^{2}-s$ or $s^{2}-s-1$. If $t=s^{2}-s-1$ then $m \neq s$.
\end{itemize}
$\mathcal{G}(\mathcal{O})$ and $\mathcal{G}(\mathcal{P}\setminus \mathcal{O})$ are both strongly regular if and only if one of the following occurs:

\begin{itemize}
\item $m=\frac{s+1}{2}$ and $t=s^{2}$.
\item $m=1,s$ and $t=s^{2}-s$.
\end{itemize}
\end{theorem}

As Theorem \ref{placeholder} shows, hemisystems seem to be the nicest and most regular examples, and if there is sufficient regularity, you are likely to get a hemisystem. This is important, because increased regularity would not be unexpected in the line-orbits of point-primitive groups. For example, suppose that one could prove that $\mathcal{G}(\mathcal{K})$ and $\mathcal{G}(\mathcal{L} \setminus \mathcal{K})$ must be strongly regular whenever $\mathcal{K}$ is the line orbit of a point-primitive group. Then Theorem \ref{placeholder} would imply that either $\mathcal{K}$ is trivial or $k=\frac{t+1}{2}$ and $s=t^{2}$ (the second part of Lemma \ref{LOL} eliminating $k=1,t$). 
This is very close to the Hemisystem Conjecture. The only difference is the added conclusion about $s=t^{2}$. However $\GQ(3,5)$ and $\LSce$ provide example of quadrangles which have point-primitive symmetry groups which have a hemisystem but which do not have order $(t^{2},t)$. So this slight extension is too much to ask for, but the idea behind it, that sufficient regularity of the $k$-cover implies that the $k$-cover is a hemisystem lends credence to the idea that something like the hemisystem conjecture might be true. Even if the conjecture as stated is not true, this makes it likely that there will be severe restrictions on what $k$ can be for line-orbits of point-primitive groups in general.

\appendix

\section{Code}\label{appendix:code}


\begin{scriptsize}

\begin{lstlisting}
# GAP code
LoadPackage("grape"); LoadPackage("atlasrep");

AtlasMaximalSubgroups := function( name )
    # uses AtlasRep to find maximal subgroups of a group stored in the Atlas
    # input: AtlasRep name
    # output: a record (maxes, group) containing a permutation group 
    #   pertaining to "name" and maximal subgroups of this permutation group 
    local tocs, gapname, numbers, maxs, wholegroup;
    if AtlasOfGroupRepresentationsInfo.remote = true  then
       tocs := AtlasTableOfContents( "remote" ).TableOfContents;
    else
       tocs := AtlasTableOfContents( "local" ).TableOfContents;
    fi;
    gapname:= First( AtlasOfGroupRepresentationsInfo.GAPnames, pair -> pair[1] = name );
    Print(gapname[3]!.nrMaxes, " maximal subgroups: ", gapname[3]!.structureMaxes, "\n");
    gapname:=gapname[2];
    numbers := List(tocs!.(gapname)!.maxes, t -> t[2]);
    maxs := List(numbers, t -> AtlasSubgroup(name, t));
    wholegroup := AtlasGroup(name);
    return rec(maxes:=maxs, group:=wholegroup);
end;

KnapsackSearch := function(Y,t)
    # Simple backtracking code to find Knapsack solutions
    # input: Y is a list of positive integers, 
    #  t is an integer that we want to find sums from Y for
    # output: characteristic vectors for the collected list of Y
    local L,U,ExtendSearch;
    L := Collected(Y);
    U:=[];
    ExtendSearch:=function(S,a,j)
        local x,y,AA,SS,k,i;
        for i in [j..Size(S)] do
            x:=S[i][1]; y:=S[i][2];
            if y <= 0 then
                continue;
            else
                if x > a then break;
                elif x = a then
                    AA := [];
                    for k in [1..Size(S)] do
                        if k=i then
                            Add(AA,L[i][2]-y+1);
                        else
                            Add(AA,L[k][2]-S[k][2]);
                        fi;
                    od;
                    Add(U,StructuralCopy(AA));
                    break;
                else
                    SS := StructuralCopy(S); SS[i][2] := y-1;
                    ExtendSearch(SS,a-x,i);
                fi;
            fi;
        od;
    end;
    ExtendSearch(L,t,1);
    return U;
end;

GeneratePossible_s_t := function(G,M)
    local B, D, b, divisors_b, c, j, r;
    # indices of maximal subgroups
    B := List(M, m -> IndexNC(G,m));
    # find all possible solutions in s and t
    D := [];
    for b in B do
        # proper nontrivial divisors of b
        divisors_b := Filtered(DivisorsInt(b), t -> 1 < t and t < b);
        c := [];
        for j in divisors_b do
            r := ((b/j)-1)/(j-1);
            if IsInt(r) and r <= (j-1)^2 and (j-1) <= r^2 and 
                RemInt( (j-1)*r*((j-1)*r+1), j-1+r ) = 0 then
               Add(c,[j-1,r]);
            fi;
        od;
        Add(D, c);
    od;
    return D; # indexed by maximal subgroup
end;

OrdersOfElementsTest := function(G, M, s, t)
    # We check that Lemma 2.2 is satisfied
    local l;
    l := (1+t)*(1+s*t);
    return ForAny(PrimeDivisors( Order(G) ), q -> 
    	(RemInt(Size(M),q)=0 or RemInt(l,q)<>0) and (q > 1+s and q > 1+t) );
end;

FilterByOrdersOfElementsTest := function(G, M, D)
    # Remove elements of D that are invalid
    local i, e, Dcopy;
    Dcopy := StructuralCopy(D);;
    for i in [1..Length(D)] do
        for e in D[i] do
            if OrdersOfElementsTest(G, M[i], e[1], e[2]) then
                Remove(Dcopy, Position(Dcopy,e)); 
            fi;
        od;
    od;
    return Dcopy;
end;

LineOrbitsTest := function(B, s, t)
    # check that we have a k-cover induced
    local k, L, POS, b, K;
    k := 1 + s*t;
    L := (t + 1) * (s*t + 1);
    POS := [];
    for b in B do
        K := LcmInt(b,k);
        if K <= L then
           Add(POS,K/k);
        fi;
    od;
    return not IsEmpty(KnapsackSearch(POS,t+1));
end;

FilterByLineOrbitsTest := function(G, M, D)
    local d, e, Dcopy, B;
    B := List(M, m -> IndexNC(G,m));
    Dcopy := StructuralCopy(D);;
    for d in Dcopy do
        for e in d do
            if not LineOrbitsTest(B, e[1], e[2]) then
                Remove(d,Position(d,e));
            fi;
        od;
    od;
    return Dcopy;
end;

MakeCollinearityGraphs := function(G, M, D)
    # check that there is a sum of subdegrees for the neighbourhood of a point,
    # then construct the collinearity graph of a possible GQ of order (s,t)
    local l, Y, d, s, t, A, FCA, S, YY, GG, a, DC, B,
    V, K, W, w, N, I, O, Inc, gamma, GlPa, FF;
    FF:=[];
    B := List(M, m -> IndexNC(G,m));
    for l in [1..Size(D)] do
        if not IsEmpty(D[l]) then
            DC := DoubleCosetRepsAndSizes(G,M[l],M[l]);
            Y := List(DC, dc -> dc[2]/Size(M[l]) );
            Remove(Y,1);
            for d in D[l] do
                s := d[1]; t := d[2];
                A := KnapsackSearch(Y,s*(t+1));
                if not IsEmpty(A) then
                    FCA := FactorCosetAction(G, M[l]);
                    S := Image(FCA, M[l]);
                    O := ShallowCopy(OrbitsDomain(S, [2..B[l]]));
                    YY := List(O,Size);
                    GG:= List(Set(YY), c -> Positions(YY,c));
                    for a in A do
                        K:=[];
                        V := List([1..Size(a)], i -> Combinations(GG[i],a[i]));
                        Add(V,List(Cartesian(K),Concatenation));
                        W := Concatenation(V);
                        for w in W do
                            N := Concatenation(O{w});
                            I := Image(FCA);
                            Inc := function(x,y) 
                                return OnPoints(y,RepresentativeAction(I,x,1)) in N and x<>y; 
                            end;
                            # construct vertex-transitive graph from putative neighbourhood
                            gamma := Graph(I, [1..B[l]], OnPoints, Inc);
                            # check if graph is strongly regular
                            if IsDistanceRegular(gamma) and Diameter(gamma)=2 then
                                GlPa:=GlobalParameters(gamma);
                                # check the graph has the right parameters
                                if GlPa=[[0,0,s*(t+1)],[1,s-1,s*t],[t+1,(s-1)*(t+1),0]] then
                                    Add(FF,[l,d,gamma,GlPa]);
                                fi;
                            fi;
                        od;
                    od;
                fi;
            od;
        fi;
    od;
    return FF;
end;

RunThroughTests := function(G, M)
    local D, graphs;
    D := GeneratePossible_s_t(G,M);;
    if ForAll(D,IsEmpty) then
        Print("No feasible orders (s,t)\n");
    else
        Print("Possible orders: ", Union(D), "\n");
        Print("... doing orders of elements tests\c");
        D := FilterByOrdersOfElementsTest(G,M,D);;
        Print(" * \n");
        if IsEmpty(D) then
            Print("Fails the orders of elements test (Lemma 2.2)\n");
        else
            Print("... doing line-orbits test\c");
            D := FilterByLineOrbitsTest(G, M, D);
            Print(" * \n");
            if IsEmpty(D) then
                Print("Fails the line-orbits k-cover test test (Lemma 2.3)\n");
            else
                graphs := MakeCollinearityGraphs(G, M, D);
                if IsEmpty(graphs) then
                    Print("No feasible collinearity graphs\n");
                else
                    return graphs;
                fi;
            fi;
        fi;
    fi;
end;

\end{lstlisting}

\end{scriptsize}

\begin{scriptsize}

\begin{lstlisting}[language=python]
# Python code

from sympy import *

indices=[97239461142009186000,5791748068511982636944259375,
439909863614532427326210000000,512372707698741056749515292734375,
16009115629875684006343550944921875,282599644298926271851701207040000000,
391965121389536908413379198941796875,1484028541986258159045049319424000000,
4050306254358548053604918389065234375,6065553341050124859256025907200000000,
147971784380684498443615773616452403200,377694424605514962329798663208960000000,
16458603283969466072643078298009600000000,69632552355255433384259177414656000000000,
2137612234906118719276348954925160732819456,4773365227577903302562875496013496320000000,
28114639032330054704286996987125956608000000,69506875251140892549372895050469742538129408,
360804534248235702038349794668116443136000000,406922407046882370719943377445244108800000000,
718237710928455889676853248854854006227337216,1227948204794415624584299721349471928320000000,
1589822867634109834649512818264086937600000000,
2672015293632648399095436193656450916024320000,
6440808214679248895891202946141582786560000000,
11133397056802701662897650806901878816768000000,
13812263671701074801477947093362577042833408000,
23847343014511647519742692273961304064000000000,
70848890361471737854803232407557410652160000000,
77933779397618911640283555648313151717376000000,
463738191456905920504166612122193960632320000000,
547294894007597532814267768386619441152000000000,
681636554498124591605978620830727274496000000000,
922293247309099546038858646725599428608000000000,
1277021419351060909899958126235445362688000000000,
4516082186421377575935948496320762111590400000000,
7870810683757187569515487092944776514764800000000,
11129716594965742092099998690932655055175680000000,
33169845024405290471529552748838701026508800000000,
49077831923864970595630460699812363763712000000000,
118131202455338139749482442245864145761075200000000,
492693551703971265784426771318116315247411200000000];

def find_s_and_t(n):
    divs=divisors(n)
    divs2=[s for s in divs if (s*(1+(s-1)**(3/2))<=n and n<=s*(1+(s-1)**3))];
    for i in range(len(divs2)):
        splus1=divs2[i]
        s=splus1-1
        if ((n/splus1-1) % s) == 0:
            t=(n/splus1-1)/s
            if ((s*t*(s*t+1)) % (s+t))==0:
                print(s,t)
    return

for n in indices:
    print(n)
    find_s_and_t(n)    
    print("----")
\end{lstlisting}

\end{scriptsize}

\section{Line orbit sizes of the known point-transitive generalised quadrangles}\label{appendix:lineorbits}

The following show all (conjugacy classes of) point-transitive subgroups of the symmetry groups of the classical polygons.

\begin{table}[!ht]
\centering
\scriptsize
\subfigure[$W(3,2)$]{
\begin{tabular}{ l c c } \toprule
Subgroup&Is Primitive?&Lengths of line-orbits\\
\midrule
$A_{5}$&$\times$&$5,10$\\
$S_{5}$&$\times$&$5,10$\\
$A_{6}$&$\checkmark$&$15$\\
$S_{6}$&$\checkmark$&$15$\\
\bottomrule \end{tabular}}
\subfigure[$W(3,3)$]{
\begin{tabular}{ l c c } \toprule
Subgroup&Is Primitive?&Lengths of line-orbits\\
\midrule
$2^{4}.5$&$\times$&$20,20$\\
$S_{5}$&$\times$&$10,30$\\
$S_{5}$&$\times$&$30,10$\\
$2^{4}.5.2$&$\times$&$20,20$\\
$2\times S_{5}$&$\times$&$30,10$\\
$2^{4}.5.4$&$\times$&$40$\\
$A_{6}$&$\times$&$10,30$\\
$S_{6}$&$\times$&$10,30$\\
$S_{6}$&$\times$&$10,30$\\
$2 \times A_{6}$&$\times$&$10,30$\\
$2^{4}.A_{5}$&$\times$&$40$\\
$2 \times S_{6}$&$\times$&$10,30$\\
$2^{4}.S_{5}$&$\times$&$40$\\
$\PSp(4,3)$&$\checkmark$&$40$\\
$\PSp(4,3).2$&$\checkmark$&$40$\\
\bottomrule \end{tabular}}
\end{table}

\begin{table}[!ht]
\centering
\scriptsize
\subfigure[$W(3,4)$]{
\begin{tabular}{ l c c } \toprule
Subgroup&Is Primitive?&Lengths of line-orbits\\
\midrule
$\PSL(2,16)$&$\times$&$17,68$\\
$\PSL(2,16).2$&$\times$&$17,68$\\
$\PSL(2,16).4$&$\times$&$17,68$\\
$\PSp(4,4)$&$\checkmark$&$85$\\
$\PSp(4,4).2$&$\checkmark$&$85$\\
\bottomrule \end{tabular}}
\subfigure[$W(3,5)$]{
\begin{tabular}{ l c c } \toprule
Subgroup&Is Primitive?&Lengths of line-orbits\\
\midrule 
$\PSL(2,25)$&$\times$&$130,26$\\
$\PSL(2,25).2$&$\times$&$130,26$\\
$2\times \PSL(2,25)$&$\times$&$130,26$\\
$2\times \PSL(2,25).2$&$\times$&$130,26$\\
$\PSp(4,5)$&$\checkmark$&$156$\\
$\PSp(4,5).2$&$\checkmark$&$156$\\
\bottomrule \end{tabular}}
\end{table}
\vspace{-0.5cm}
\begin{table}[!ht]
\centering
\scriptsize
\subfigure[$Q(4,2)$]{
\begin{tabular}{ l c c } \toprule
Subgroup&Is Primitive?&Lengths of line-orbits\\
\midrule
$A_{5}$&$\times$&$5,10$\\
$S_{5}$&$\times$&$5,10$\\
$A_{6}$&$\checkmark$&$15$\\
$S_{6}$&$\checkmark$&$15$\\
\bottomrule \end{tabular}}
\subfigure[$Q(4,3)$]{
\begin{tabular}{ l c c } \toprule
Subgroup&Is Primitive?&Lengths of line-orbits\\
\midrule
$2^{4}.5.4$&$\times$&$40$\\
$2^{4}.A_{5}$&$\times$&$40$\\
$2^{4}.S_{5}$&$\times$&$40$\\
$\PSp(4,3)$&$\checkmark$&$40$\\
$\PSp(4,3).2$&$\checkmark$&$40$\\
\bottomrule \end{tabular}}
\end{table}

\begin{table}[!ht]
\centering
\scriptsize
\subfigure[$Q(4,4)$]{
\begin{tabular}{ l c c } \toprule
Subgroup&Is Primitive?&Lengths of line-orbits\\
\midrule
$\PSL_{2}(16)$&$\times$&$68,17$\\
$\PSL_{2}(16).2$&$\times$&$68,17$\\
$\PSL_{2}(16).4$&$\times$&$68,17$\\
$\PSp_{4}(4)$&$\checkmark$&$85$\\
$\PSp_{4}(4).2$&$\checkmark$&$85$\\
\bottomrule \end{tabular}}
\subfigure[$Q(4,5)$]{
\begin{tabular}{ l c c } \toprule
Subgroup&Is Primitive?&Lengths of line-orbits\\
\midrule
$\PSp(4,5)$&$\checkmark$&$156$\\
$\PSp(4,5).2$&$\checkmark$&$156$\\
\bottomrule \end{tabular}}
\end{table}

\begin{table}[!ht]
\centering
\scriptsize
\subfigure[$H(3,4)$]{
\begin{tabular}{ l c c } \toprule
Subgroup&Is Primitive?&Lengths of line-orbits\\
\midrule
$\PSU(4,2)$&$\checkmark$&$27$\\
$\PSU(4,2).2$&$\checkmark$&$27$\\
\bottomrule \end{tabular}}
\subfigure[$H(3,9)$]{
\begin{tabular}{ c c c } \toprule
Subgroup&Is Primitive?&Lengths of line-orbits\\
\midrule
$\PSL(3,4)$&$\checkmark$&$56,56$\\
$\PSL(3,4).2$&$\checkmark$&$112$\\
$\PSL(3,4).2$&$\checkmark$&$56,56$\\
$\PSL(3,4).2$&$\checkmark$&$112$\\
$\PSL(3,4).2^{2}$&$\checkmark$&$112$\\
$\PSU(4,3)$&$\checkmark$&$112$\\
$\PSU(4,3).2$&$\checkmark$&$112$\\
$\PSU(4,3).2$&$\checkmark$&$112$\\
$\PSU(4,3).2$&$\checkmark$&$112$\\
$\PSU(4,3).4$&$\checkmark$&$112$\\
$\PSU(4,3).2^{2}$&$\checkmark$&$112$\\
$\PSU(4,3).2^{2}$&$\checkmark$&$112$\\
$\PSU(4,3).D_{8}$&$\checkmark$&$112$\\
\bottomrule \end{tabular}}
\end{table}

\begin{table}[!ht]
\centering
\scriptsize
\subfigure[$Q^{-}(5,2)$]{
\begin{tabular}{ l c c } \toprule
Subgroup&Is Primitive?&Lengths of line-orbits\\
\midrule
$3^{2}.3$&$\times$&$9,9,9,9,9$\\
$9.3$&$\times$&$9,27,9$\\
$3^{2}.3.2$&$\times$&$9,9,9,9,9$\\
$3^{2}.6$&$\times$&$9,9,9,18$\\
$9.6$&$\times$&$9,9,27$\\
$3^{3}.3$&$\times$&$9,27,9$\\
$3^{2}.3.4$&$\times$&$9,18,18$\\
$3^{2}.3.(2\times2)$&$\times$&$9,9,9,18$\\\
$3^{3}.3.2$&$\times$&$27,9,9$\\
$3^{3}.3.2$&$\times$&$9,27,9$\\
$3^{3}.3.2$&$\times$&$9,27,9$\\
$3^{3}.3.Q_{8}$&$\times$&$9,36$\\
$3^{3}.3.8$&$\times$&$9,36$\\
$3^{3}.(2\times2).3$&$\times$&$9,36$\\
$3^{3}.(2\times2).3.2$&$\times$&$9,18,18$\\
$3^{3}.3.(2\times2)$&$\times$&$27,18$\\
$3^{3}.3.QD_{16}$&$\times$&$27,9,9$\\
$3^{3}.(2\times2).3.2$&$\times$&$9,36$\\
$3^{3}.(2\times2).3.2$&$\times$&$18,27$\\
$(S_{3} \times S_{3} \times S_{3}).3$&$\times$&$27,18$\\
$3^{3}.3.Q_{8}.3$&$\times$&$9,36$\\
$3^{3}.3.Q_{8}.3.2$&$\times$&$36,9$\\
$3^{3}.(2\times2).3.2.2$&$\times$&$27,18$\\
$\PSU(4,2)$&$\checkmark$&$45$\\
$\PSU(4,2).2$&$\checkmark$&$45$\\
\bottomrule \end{tabular}}
\subfigure[$Q^{-}(5,3)$]{
\begin{tabular}{ l c c } \toprule
Subgroup&Is Primitive?&Lengths of line-orbits\\
\midrule
$4\times \PSL(3,2)$&$\times$&$84,168,28$\\
$2\times (\PSL(3,2).2)$&$\times$&$84,168,28$\\
$4\times (\PSL(3,2).2)$&$\times$&$84,168,28$\\

$\PSU(3,3)$&$\times$&$252,28$\\
$2\times \PSU(3,3)$&$\times$&$252,28$\\
$\PSU(3,3).2$&$\times$&$252,28$\\
$\PSU(3,3).2$&$\times$&$252,28$\\
$2\times(\PSU(3,3).2)$&$\times$&$252,28$\\
$4\times \PSU(3,3)$&$\times$&$252,28$\\
$2\times(\PSU(3,3).2)$&$\times$&$252,28$\\
$\PSU(3,3).D_{8}$&$\times$&$252,28$\\

$\PSL(3,4).2$&$\times$&$280$\\
$\PSL(3,4).2$&$\times$&$280$\\
$\PSL(3,4).4$&$\times$&$280$\\

$\PSU(4,3)$&$\checkmark$&$280$\\
$\PSU(4,3).2$&$\checkmark$&$280$\\
$\PSU(4,3).2$&$\checkmark$&$280$\\
$\PSU(4,3).2$&$\checkmark$&$280$\\
$\PSU(4,3).4$&$\checkmark$&$280$\\
$\PSU(4,3).2^{2}$&$\checkmark$&$280$\\
$\PSU(4,3).2^{2}$&$\checkmark$&$280$\\
$\PSU(4,3).D_{8}$&$\checkmark$&$280$\\
\bottomrule \end{tabular}}
\end{table}

\begin{table}[!ht]
\centering
\scriptsize
\subfigure[$H(4,4)$]{
\begin{tabular}{ l c c } \toprule
Subgroup&Is Primitive?&Lengths of line-orbits\\
\midrule
$\PSU(5,2)$&$\checkmark$&$297$\\
$\PSU(5,2).2$&$\checkmark$&$297$\\
\bottomrule \end{tabular}}
\subfigure[$H(4,4)^{D}$]{
\begin{tabular}{ l c c } \toprule
Subgroup&Is Primitive?&Lengths of line-orbits\\
\midrule
$\PSU(5,2)$&$\checkmark$&$165$\\
$\PSU(5,2).2$&$\checkmark$&$165$\\
\bottomrule \end{tabular}}
\end{table}


\newpage

\begin{thebibliography}{99}

\bibitem{fining}
J.~Bamberg, A.~Betten, P.~Cara, J.~De~Beule, M.~Lavrauw, and M.~Neunh\"offer.
\newblock {\em {FinInG -- Finite Incidence Geometry, Version 1.4.1}}, 2018.

\bibitem{Bamberg:2012yf}
J.~Bamberg, M.~Giudici, J.~Morris, G.~F. Royle, and P.~Spiga.
\newblock Generalised quadrangles with a group of automorphisms acting
  primitively on points and lines.
\newblock {\em J. Combin. Theory Ser. A}, 119(7):1479--1499, 2012.

\bibitem{PsuedoHyperovals}
J.~Bamberg, S.~Glasby, T.~Popiel, and C.~Praeger.
\newblock Generalized quadrangles and transitive pseudo-hyperovals.
\newblock {\em Journal of Combinatorial Designs}, 24(4):151--164, 3 2016.

\bibitem{Bamberg:2017ab}
J.~Bamberg, T.~Popiel, and C.~E. Praeger.
\newblock Point-primitive, line-transitive generalised quadrangles of holomorph
  type.
\newblock {\em J. Group Theory}, 20(2):269--287, 2017.

\bibitem{Bamberg2018Quadrangles}
J.~Bamberg, T.~Popiel, and C.~E. Praeger.
\newblock Simple groups, product actions, and generalized quadrangles.
\newblock {\em Nagoya Mathematical Journal}, pages 1--40, 2017.

\bibitem{Bray:2013aa}
J.~N. Bray, D.~F. Holt, and C.~M. Roney-Dougal.
\newblock {\em The maximal subgroups of the low-dimensional finite classical
  groups}, volume 407 of {\em London Mathematical Society Lecture Note Series}.
\newblock Cambridge University Press, Cambridge, 2013.
\newblock With a foreword by Martin Liebeck.

\bibitem{DistanceRegularGraphs}
A.~E. Brouwer, A.~M. Cohen, and A.~Neumaier.
\newblock {\em Distance-Regular Graphs}.
\newblock A series of modern surveys in mathematics. Springer-Verlag Berlin
  Heidelberg, 1989.

\bibitem{Buekenhout:1993bf}
F.~Buekenhout and H.~Van~Maldeghem.
\newblock Remarks on finite generalized hexagons and octagons with a
  point-transitive automorphism group.
\newblock In {\em Finite geometry and combinatorics ({D}einze, 1992)}, volume
  191 of {\em London Math. Soc. Lecture Note Ser.}, pages 89--102. Cambridge
  Univ. Press, Cambridge, 1993.

\bibitem{Buekenhout:1994zp}
F.~Buekenhout and H.~Van~Maldeghem.
\newblock Finite distance-transitive generalized polygons.
\newblock {\em Geom. Dedicata}, 52(1):41--51, 1994.

\bibitem{Cameron:1997aa}
P.~J. Cameron.
\newblock Finite geometry after {A}schbacher's theorem: {${\rm PGL}(n,q)$} from
  a {K}leinian viewpoint.
\newblock In {\em Geometry, combinatorial designs and related structures
  ({S}petses, 1996)}, volume 245 of {\em London Math. Soc. Lecture Note Ser.},
  pages 43--61. Cambridge Univ. Press, Cambridge, 1997.

\bibitem{ATLAS}
J.~H. Conway, R.~T. Curtis, S.~P. Norton, R.~A. Parker, and R.~A. Wilson.
\newblock {\em Atlas of finite groups}.
\newblock Oxford University Press, Eynsham, 1985.
\newblock Maximal subgroups and ordinary characters for simple groups, With
  computational assistance from J. G. Thackray.

\bibitem{Dembowski:1997sy}
P.~Dembowski.
\newblock {\em Finite geometries}.
\newblock Classics in Mathematics. Springer-Verlag, Berlin, 1997.
\newblock Reprint of the 1968 original.

\bibitem{Feit:1964os}
W.~Feit and G.~Higman.
\newblock The nonexistence of certain generalized polygons.
\newblock {\em J. Algebra}, 1:114--131, 1964.

\bibitem{FongSeitz}
P.~Fong and G.~M. Seitz.
\newblock Groups with a (b,n)-pair of rank 2.
\newblock In T.~Gagen, M.~P. Hale, and E.~E. Shult, editors, {\em Finite Groups
  '72}, volume~7 of {\em North-Holland Mathematics Studies}, pages 36 -- 40.
  North-Holland, 1973.

\bibitem{GAP4}
The GAP~Group.
\newblock {\em {GAP -- Groups, Algorithms, and Programming, Version 4.10.0}},
  2018.

\bibitem{GiudiciTransitiveSubgroups}
M.~Giudici, S.~P. Glasby, and C.~E. Praeger.
\newblock Subgroups of classical groups that are transitive on subspaces.
\newblock 2020.

\bibitem{KantorPrimitivePlanes}
W.~M. Kantor.
\newblock Flag-transitive planes.
\newblock In {\em Finite geometries ({W}innipeg, {M}an., 1984)}, volume 103 of
  {\em Lecture Notes in Pure and Appl. Math.}, pages 179--181. Dekker, New
  York, 1985.

\bibitem{KleidmanLiebeck}
P.~Kleidman and M.~W. Liebeck.
\newblock {\em The subgroup structure of the finite classical groups}.
\newblock Cambridge University Press, Cambridge, 1990.

\bibitem{PayneThas:2009}
S.~E. Payne and J.~A. Thas.
\newblock {\em Finite Generalised Quadrangles}.
\newblock European Mathematical Society, second edition, 2009.

\bibitem{PraegerInclusions}
C.~E. Praeger.
\newblock The inclusion problem for finite primitive permutation groups.
\newblock {\em Proceedings of the London Mathematical Society},
  s3-60(1):68--88, 1990.

\bibitem{Schneider:2008lr}
C.~Schneider and H.~Van~Maldeghem.
\newblock Primitive flag-transitive generalized hexagons and octagons.
\newblock {\em J. Combin. Theory Ser. A}, 115(8):1436--1455, 2008.

\bibitem{SegreHemisystems}
B.~Segre.
\newblock Forme e geometrie hermitiane con particolare rigwrdo al case finito.
\newblock {\em Annali di Matematica Pura ed Applicata}, 4(70):1--202, 1965.

\bibitem{Design}
L.~H. Soicher.
\newblock {\em The DESIGN package for GAP, Version 1.6}, 2011.

\bibitem{GRAPE}
L.~H. Soicher.
\newblock {\em The GRAPE package for GAP, Version 4.8.1}, 2018.

\bibitem{THAS1989103}
J.~Thas.
\newblock Interesting pointsets in generalized quadrangles and partial
  geometries.
\newblock {\em Linear Algebra and its Applications}, 114-115:103 -- 131, 1989.
\newblock Special Issue Dedicated to Alan J. Hoffman.

\bibitem{Tits:1959cl}
J.~Tits.
\newblock Sur la trialit{\'e} et certains groupes qui s'en d{\'e}duisent.
\newblock {\em Inst. Hautes {\'E}tudes Sci. Publ. Math.}, (2):13--60, 1959.

\bibitem{TitsBuildings}
J.~Tits.
\newblock {\em Buildings of Spherical Type and Finite BN-Pairs}, volume 386 of
  {\em Lecture Notes in Mathematics}.
\newblock Springer-Verlag Berlin Heidelberg, 1974.

\bibitem{Van-Maldeghem:1998sj}
H.~Van~Maldeghem.
\newblock {\em Generalized polygons}.
\newblock Modern Birkh{\"a}user Classics. Birkh{\"a}user/Springer Basel AG,
  Basel, 1998.
\newblock [2011 reprint of the 1998 original] [MR1725957].

\bibitem{MaximalSubgroups}
R.~A. {Wilson}.
\newblock Maximal subgroups of sporadic groups.
\newblock {\em arXiv e-prints}, page arXiv:1701.02095, Jan 2017.

\bibitem{AtlasRep1.5.1}
R.~A. Wilson, R.~A. Parker, S.~Nickerson, J.~N. Bray, and T.~Breuer.
\newblock {\em {AtlasRep}, A \textsf{GAP} Interface to the Atlas of Group
  Representations, Version 1.5.1}, March 2016.
\newblock \textsf{GAP} package.

\end{thebibliography}

\end{document}